\documentclass[a4paper,12pt]{article}
\usepackage[T1]{fontenc}
\usepackage{mathrsfs}
\usepackage{graphicx}
\usepackage{enumerate}
\usepackage{multicol}
\usepackage{color}
\usepackage{amsmath,amssymb}
\usepackage{amsthm}
\usepackage{textcomp}
\usepackage{mathabx}
\usepackage{hyperref}

\usepackage{times}
\usepackage[varg]{txfonts}

\usepackage[top=1in, bottom=1in, left=0.75in, right=0.75in]{geometry}

\theoremstyle{plain}
\newtheorem{thm}{Theorem}
\newtheorem{lem}{Lemma}

\newtheorem{prop}{Proposition}

\theoremstyle{definition}
\newtheorem{dfn}{Definition}

\theoremstyle{remark}
\newtheorem{rem}{Remark}

\newtheorem*{ackn}{Acknowledgment}

\renewcommand{\lneq}{<}

\title{Operator relations characterizing higher-order differential operators}
\author{W{\l}odzimierz Fechner, Eszter Gselmann and Aleksandra \'{S}wi\k{a}tczak}

\begin{document}

\maketitle

\begin{abstract}
 Let $r$ be a positive integer, $N$ be a nonnegative integer and $\Omega \subset \mathbb{R}^{r}$ be a domain. Further, for all multi-indices $\alpha \in \mathbb{N}^{r}$, $|\alpha|\leq N$, let us consider the partial differential operator $D^{\alpha}$ defined by 
\[
 D^{\alpha}= \frac{\partial^{|\alpha|}}{\partial x_{1}^{\alpha_{1}}\cdots \partial x_{r}^{\alpha_{r}}}, 
\]
where $\alpha= (\alpha_{1}, \ldots, \alpha_{r})$. Here by definition we mean $D^{0}\equiv \mathrm{id}$.  
An easy computation shows that if $f, g\in \mathscr{C}^{N}(\Omega)$ and $\alpha \in \mathbb{N}^{r}, |\alpha|\leq N$, then we have 
\[
\tag{$\ast$}
D^{\alpha}(f\cdot g) = \sum_{\beta\leq \alpha}\binom{\alpha}{\beta}D^{\beta}(f)\cdot D^{\alpha - \beta}(g). 
\]
This paper is devoted to the study of identity $(\ast)$ in the space $\mathscr{C}(\Omega)$. More precisely, if $r$ is a positive integer, $N$ is a nonnegative integer and $\Omega \subset \mathbb{R}^{r}$ is a domain, then we describe those mappings $T_{\alpha} \colon \mathscr{C}(\Omega)\to \mathscr{C}(\Omega)$, $\alpha \in \mathbb{N}^{r}, |\alpha|\leq N$ that satisfy identity $(\ast)$ for all possible multi-indices $\alpha\in \mathbb{N}^{r}$, $|\alpha|\leq N$. Our main result says that if the domain is $\mathscr{C}(\Omega)$, then the mappings $T_{\alpha}$ are of a rather special form. Related results in the space $\mathscr{C}^{N}(\Omega)$ are also presented. 
\end{abstract}

\section{Introduction and preliminaries}

In this paper the set of real numbers is denoted by $\mathbb{R}$, the set of complex numbers by $\mathbb{C}$, and the set of nonnegative integers by $\mathbb{N}$. 

Let $r$ be a positive integer. Elements of $\mathbb{N}^{r}$ will be called $r$-dimensional multi-indices. Sums and differences of multi-indices (of the same dimension) are meant to be componentwise, i.e., if $\alpha, \beta \in \mathbb{N}^{r}$, then 
\[
	\alpha \pm \beta= (\alpha_1 \pm \beta_1,\ldots, \alpha_r \pm \beta_r)
	\]
Further, if $\alpha, \beta \in \mathbb{N}^{r}$, then we write $\alpha \leq \beta$ if for all $i=1, \ldots, r$ we have $\alpha_{i}\leq \beta_{i}$, where $\alpha= (\alpha_{1}, \ldots, \alpha_{r})$ and $\beta= (\beta_{1}, \ldots, \beta_{r})$. If for the multi-indices $\alpha, \beta \in \mathbb{N}^{r}$ we have $\alpha \leq \beta$ and $\alpha \neq \beta$, we will write 
$\alpha \lneq \beta$. 
By the height of a multi-index $\alpha \in \mathbb{N}^{r}$ we understand $|\alpha|= \sum_{i=1}^{r}\alpha_{i}$, where $\alpha= (\alpha_{1}, \ldots, \alpha_{r})$. 
Finally, we will also use the notion of factorial and binomial coefficients in this multi-index setting. If $\alpha, \beta  \in \mathbb{N}^{r}$, then 
\[
 \alpha ! = \alpha_1! \cdots  \alpha_r!
\]
and 
\[
 \binom{\alpha}{\beta} = \binom{\alpha_1}{\beta_1}\cdots\binom{\alpha_r}{\beta_r}=\frac{\alpha!}{\beta!(\alpha-\beta)!}, 
\]
where $\alpha= (\alpha_{1}, \ldots, \alpha_{r})$ and $\beta = (\beta_{1}, \ldots, \beta_{r})$.

Let $r$ be a positive integer, $N$ be a nonnegative integer and $\Omega \subset \mathbb{R}^{r}$ be a domain (i.e. a nonempty, open and connected set). For all multi-indices $\alpha \in \mathbb{N}^{r}$, $|\alpha|\leq N$, let us consider the partial differential operator $D^{\alpha}$ defined by 
\[
 D^{\alpha}= \frac{\partial^{|\alpha|}}{\partial x_{1}^{\alpha_{1}}\cdots \partial x_{r}^{\alpha_{r}}}, 
\]
where $\alpha= (\alpha_{1}, \ldots, \alpha_{r})$. Here by definition we mean $D^{0}= \mathrm{id}$.  Let further 
\[
 \mathscr{C}^{N}(\Omega)= 
 \left\{  f\colon \Omega \to \mathbb{R}\, \vert \, f \text{ is } N \text{ times continuously differentiable}\right\}. 
\]

An easy computation shows that if $f, g\in \mathscr{C}^{N}(\Omega)$ and $\alpha \in \mathbb{N}^{r}, |\alpha|\leq N$, then we have 
\begin{equation}\label{Eq_diff}
D^{\alpha}(f\cdot g) = \sum_{\beta\leq \alpha}\binom{\alpha}{\beta}D^{\beta}(f)\cdot D^{\alpha - \beta}(g). 
\end{equation}

The main aim of this paper will be about the converse in some sense. More precisely, in this paper, we will study the solutions $T_{\alpha}\colon \mathscr{C}^{N}(\Omega)\to \mathscr{C}(\Omega)$ of 
the operator equation 
\[
T_{\alpha}(f\cdot g)=\sum_{\beta\leq \alpha} \binom{\alpha}{\beta} T_{\beta}(f)T_{\alpha-\beta}(g)
\qquad 
\left(f, g\in \mathscr{C}^{N}(\Omega)\right)
\]
for all multi-indices $\alpha \in \mathbb{N}^{r}$, $|\alpha|\leq N$.

Equations analogous to \eqref{Eq_diff} have an important role not only in connection to characterization theorems related to differential operators but also in harmonic and spectral analysis. 

In the following, we will use the notations and the terminology of the monographs Sz\'{e}kelyhidi \cite{Sze06, Sze14} and while considering moment sequences of higher rank, the terminology of \cite{FecGseSze21}.

Let $(G,\cdot)$ be an Abelian semigroup. A nonzero function $m\colon G\to \mathbb{C}$ is called {\it exponential}, if 
\[
m(x\cdot y)=m(x)m(y)
\] 
holds for all $x,y$ in $G$. 
Let $N$ be a nonnegative integer. A function $\varphi\colon G\to \mathbb{C}$ is termed to be a \emph{ moment function of order} $N$, if there exist functions $\varphi_k\colon G \to \mathbb{C}$ such that $\varphi_0=1$, $\varphi_N=\varphi$ and
\begin{equation}
\varphi_k(x\cdot y)=\sum_{j=0}^k \binom{k}{j} \varphi_j(x)\varphi_{k-j}(y)
\label{Eq_moment}
\end{equation} 
for all $x$ and $y$ in $G$ and $k=0,1,\dots,N$. If $G$ is a monoid with the neutral element $1$, then this concept can be extended by relaxing the assumption \hbox{$\varphi_0\equiv 1$} to $\varphi_0(1)=1$. 
In this case, $\varphi_0$ is an arbitrary exponential function and we say that $\varphi_0$ {\it generates the generalized moment sequence of order} $N$ and the function $\varphi_k$ is a {\it generalized moment function of order} $k$, or, if we want to specify the exponential $\varphi_0$, then we say that $\varphi_k$ is a {\it generalized moment function of order} $k$ {\it associated with the exponential} $\varphi_0$.

\begin{dfn}
Let $G$ be an Abelian semigroup, $r$ a positive integer, and for each multi-index $\alpha$ in $\mathbb{N}^r$ 
let $f_{\alpha}\colon G\to \mathbb{C}$ be a function. We say that $(f_{\alpha})_{\alpha \in \mathbb{N}^{r}}$ is a \emph{generalized 
moment sequence of rank $r$}, if 
\begin{equation}\label{Eq3}
f_{\alpha}(x\cdot y)=\sum_{\beta\leq \alpha} \binom{\alpha}{\beta} f_{\beta}(x)f_{\alpha-\beta}(y)
\end{equation}
holds whenever $x,y$ are in $G$. The function $f_{\mathbf{0}}$, where $\mathbf{0}$ is the zero element in $\mathbb{N}^r$, is called the \emph{ generating function} of the sequence.
\end{dfn}

\begin{rem}
 For $r=1$, instead of multi-indices, we have nonnegative integer indices. Thus generalized moment functions of rank $1$ are simply moment sequences. 
\end{rem}

\begin{rem}
Assume now that $(G, \cdot)$ is an Abelian group (not only a semigroup). 
For $\alpha=\mathbf{0}$ we have
\[
f_{ \mathbf{0}}(x \cdot y)=f_{\mathbf{0}}(x)\cdot f_{\mathbf{0}}(y) 
\]
for each $x, y$ in $G$, hence $f_{\mathbf{0}}=m$ is always an exponential, or identically zero. It can be proved by induction on the height of the multi-index $\alpha\in \mathbb{N}^{r}$ that if $f_{\mathbf{0}}$ is the identically zero function, then for all multi-index $\alpha$, the mapping $f_{\alpha}$ must be identically zero, too. 
\end{rem}

In a rather natural way, the above notions can be extended from complex-valued mappings to mappings whose range is a (commutative) ring. 
Indeed, if $(G, \cdot)$ is an Abelian semigroup and $Q$ is a commutative ring, $r$ is a positive integer, and $\alpha\in\mathbb{N}^r$ is a multi-index, then a function $f\colon G\to Q$ is a generalized moment function of rank $r$ and of order $N$, where $N=|\alpha|$, if for all multi-indices $\beta \in \mathbb{N}^{r}$ with $|\beta| \leq N$, there exists a function $f_{\beta}\colon G\to Q$ such that $f= f_{\alpha}$ and we have 
\begin{equation}
f_{\beta}(x\cdot y)=\sum_{\gamma\leq \beta} \binom{\beta}{\gamma} f_{\gamma}(x)f_{\beta-\gamma}(y)
\end{equation}
holds whenever $x,y\in G$ and for all multi-indices $\beta \in \mathbb{N}^{r}$ with $|\beta |\leq N$.

\begin{rem}\label{rem3}
 Using the above definition this means that if $N\geq 1$ and we consider $\mathscr{C}^{N}(\Omega)$ with the pointwise product of functions, then it will become an Abelian semigroup and we take $\mathscr{C}(\Omega)$ to be the range, then the sequence of mappings $(D^{\alpha})_{|\alpha|\leq N}$ forms a moment sequence of rank $r$. 
\end{rem}

\section{Characterizations of higher order differential operators}

The main aim of this paper is to investigate the following problem: 
Let $r$ be a positive integer, $N$ be a nonnegative integer and $\Omega \subset \mathbb{R}^{r}$ be a domain (i.e. a nonempty, open and connected set). Determine the mappings $T_{\alpha} \colon \mathscr{C}^{N}(\Omega)\to \mathscr{C}(\Omega)$, $\alpha \in \mathbb{N}^{r}, |\alpha|\leq N$ if they fulfill 
\begin{equation}\label{Eq_main}
 T_{\beta}(f\cdot g)=\sum_{\gamma\leq \beta} \binom{\beta}{\gamma} T_{\gamma}(f)T_{\beta-\gamma}(g)
\end{equation}
for all $f, g\in \mathscr{C}^{N}(\Omega)$ and for all multi-indices $\beta \in \mathbb{N}^{r}$, $|\beta|\leq N$.

Observe that if $\beta= \mathbf{0}=(0, \ldots, 0)$, then the above identity becomes 
\[
 T_{\mathbf{0}}(f\cdot g)= T_{\mathbf{0}}(f)\cdot T_{\mathbf{0}}(g) 
  \qquad 
  \left(f, g\in \mathscr{C}^{N}(\Omega)\right). 
\]
This means, that similarly to the group case, the first element of the sequence, i.e. $T_{\mathbf{0}}$  is an `exponential'. 

Recall again that if $(G,\cdot)$ is an Abelian \emph{group}, then a nonzero function $m\colon G\to \mathbb{C}$ is an exponential, if 
\[
m(x\cdot y)=m(x)m(y)
\] 
holds for all $x,y$ in $G$. In the case of this concept, the fact that the range of $m$ is the set of complex numbers plays a key role. Indeed, if $m$ is an exponential on the Abelian group $G$, then either $m$ is identically zero, or nowhere zero. At the same time, as we will see below, analogous statements are not true for mappings $T\colon \mathscr{C}^{N}(\Omega)\to \mathscr{C}(\Omega)$. 

The study of multiplicative maps between function spaces has quite extensive literature. Here we quote only two of them, but the interested reader can consult e.g. Artstein-Avidan--Faifman--Milman  \cite{ArtFaiMil12},  Milgram \cite{Mil49}, Mr\v{c}un \cite{Mrc05} and Mr\v{c}un--\v{S}emrl \cite{MrcSem07}.

A result from \cite{MrcSem07} concerning \emph{bijective} multiplicative mappings between the function spaces $\mathscr{C}(X)$ and $\mathscr{C}(Y)$ says that
if we are given compact Hausdorff spaces $X$ and $Y$, $\tau \colon Y\to X$ is a homeomorphism and $p\in \mathscr{C}(Y)$ is a positive function, then the mapping $\mathcal{M}\colon \mathscr{C}(X)\to \mathscr{C}(Y)$ defined by 
 \[
  \mathcal{M}(f)(x)= \left|f(\tau(x)) \right|^{p(x)}\cdot \mathrm{sgn}\left(f(\tau(x))\right)
  \qquad 
  \left(x\in Y, f\in \mathscr{C}(X)\right)
 \]
is a bijective and multiplicative map, i.e. we have 
\[
 \mathcal{M}(f\cdot g)(x)= \mathcal{M}(f)(x) \cdot \mathcal{M}(g)(x)
\]
for all $x\in Y$ and $f, g\in \mathscr{C}(X)$. 

In view of this, if $K\subset \mathbb{R}^{r}$ is a \emph{compact} set and $\tau \colon K \to K$ is a homeomorphism, then 
the mapping $\mathcal{M}\colon \mathscr{C}(K)\to \mathscr{C}(K)$ defined by 
 \[
  \mathcal{M}(f)(x)= \left|f(\tau(x)) \right|^{p(x)}\cdot \mathrm{sgn}\left(f(\tau(x))\right)
  \qquad 
  \left(x\in K, f\in \mathscr{C}(K)\right)
 \]
is a bijective and multiplicative map. Firstly observe that this is only one direction and not an `if and only if' statement. Further, in general, we intend to work on a \emph{domain} $\Omega \subset \mathbb{R}^{r}$ and we cannot a priori assume that the mapping in question is \emph{bijective}. 

A corollary of a result from Mr\v{c}un \cite{Mrc05} describes \emph{bijective} multiplicative self-mappings of $\mathscr{C}^{N}(\Omega)$, where $N$ is a fixed positive integer. 
 Let $N, r$ be a positive integers and $\Omega  \subset \mathbb{R}^{r}$ be a $\mathscr{C}^{N}$-manifold. Then for any multiplicative bijection $\mathcal{M}\colon \mathscr{C}^{N}(\Omega)\to \mathscr{C}^{N}(\Omega)$ there exists a unique $\mathscr{C}^{N}$-diffeomorphism $\tau\colon \Omega \to \Omega$ such that 
 \[
  \mathcal{M}(f)(x)= f(\tau(x)) 
  \qquad 
  \left(x\in \Omega, f\in \mathscr{C}^{N}(\Omega)\right)
 \]
holds.

In the cases we intend to study, unfortunately, the range of the mappings is not $\mathscr{C}^{N}(\Omega)$, but the much larger function space $\mathscr{C}(\Omega)$. In addition, in general, it cannot be guaranteed that the mapping $T_{\mathbf{0}}$ is bijective. However, without the assumption of bijectivity, we cannot expect to be able to describe the multiplicative mappings in these spaces. Thus, in this paper, we will determine the moment functions of the spaces in question in the case of some important multiplicative mappings.

\subsection{A non-bijective case}

Let $r$ be a positive integer, $N$ be nonnegative a integer, and $\Omega\subset \mathbb{R}^{r}$ be a domain. Then the mapping $T_{\mathbf{0}}\colon \mathscr{C}^{N}(\Omega)\to \mathscr{C}(\Omega)$ defined by 
\[
 T_{\mathbf{0}}(f)(x)= 1 
 \qquad 
 \left(x\in \Omega, f\in \mathscr{C}^{N}(\Omega)\right)
\]
is multiplicative (and non-bijective). Therefore, it can be suitable to generate a moment sequence. As we will see, this mapping generates a fairly trivial moment sequence.

\begin{thm}
 Let $r$ be a positive integer, $N$ be nonnegative a integer, and $\Omega\subset \mathbb{R}^{r}$ be a domain. Assume further that for all multi-indices $\alpha\in \mathbb{N}^{r}$, $|\alpha|\leq N$, we are given a mapping $T_{\alpha}\colon \mathscr{C}^{N}(\Omega)\to \mathscr{C}(\Omega)$ such that 
 \[
 T_{\mathbf{0}}(f)(x)= 1 
 \qquad 
 \left(x\in \Omega, f\in \mathscr{C}^{N}(\Omega)\right)
\]
and $(T_{\alpha})_{|\alpha|\leq N}$ forms a moment sequence of rank $r$ and of order $N$. Then for all multi-indices $\alpha\in \mathbb{N}^{r}$ with $\alpha\neq \mathbf{0}$ and $|\alpha|\leq N$ we have 
\[
 T_{\alpha}(f)(x)= 0
\]
for all $x\in \Omega$ and $f\in \mathscr{C}^{N}(\Omega)$. 
\end{thm}

\begin{proof}
 We prove the statement on induction of the height of the multi-index $\alpha\in \mathbb{N}^{r}$. 
 Accordingly, let $\alpha\in \mathbb{N}^{r}$ be an arbitrary multi-index with $|\alpha|=1$. 
 Then 
 \[
  T_{\alpha}(f\cdot g)= T_{\mathbf{0}}(f)T_{\alpha}(g)+T_{\alpha}(f)T_{\mathbf{0}}(g)
 \]
holds for all $f, g\in \mathscr{C}^{N}(\Omega)$. Since 
\[
 T_{\mathbf{0}}(f)(x)= 1 
 \qquad 
 \left(x\in \Omega, f\in \mathscr{C}^{N}(\Omega)\right), 
\]
this means that 
\[
  T_{\alpha}(f\cdot g)= T_{\alpha}(f)+T_{\alpha}(g)
 \]
all $f, g\in \mathscr{C}^{N}(\Omega)$. Let $f$ and $g$ be the identically zero functions on $\Omega$, we get that 
\[
 T_{\alpha}(0)=2T_{\alpha}(0), 
\]
so $T_{\alpha}(0)=0$. This however yields that 
\[
 T_{\alpha}(f\cdot 0)= T_{\alpha}(f)+T_{\alpha}(0)
\]
for all $f\in \mathscr{C}^{N}(\Omega)$. Thus 
\[
 T_{\alpha}(f)(x)=0
\]
for all $f\in \mathscr{C}^{N}(\Omega)$ and $x\in \Omega$. 

Let now $k \in \left\{ 1, \ldots, N-1\right\}$ be arbitrary, and suppose that 
\[
 T_{\beta}(f)(x)= 0 
 \qquad 
 \left(f\in \mathscr{C}^{N}(\Omega), x\in \Omega\right)
\]
holds for all multi-indices $\beta$ with $|\beta|\leq k$. Let further $\alpha\in \mathbb{N}^{r}$ be an arbitrary multi-index with $|\alpha|=k+1$. Then 
\begin{align*}
 T_{\alpha}(f \cdot g)&= 
 \sum_{\beta \leq \alpha}\binom{\alpha}{\beta}T_{\beta}(f)\cdot T_{\alpha-\beta}(g)\\
 &=  T_{\mathbf{0}}(f)T_{\alpha}(g)+T_{\alpha}(f)T_{\mathbf{0}}(g) 
 + \sum_{0 \lneq \beta \lneq \alpha} \binom{\alpha}{\beta}T_{\beta}(f)\cdot T_{\alpha-\beta}(g)\\
 &= T_{\alpha}(f)+T_{\alpha}(g)
\end{align*}
holds for all $f, g\in \mathscr{C}^{N}(\Omega)$. This is exactly the same equation that we solved above. Thus 
\[
 T_{\alpha}(f)(x)=0
\]
for all $f\in \mathscr{C}^{N}(\Omega)$ and $x\in \Omega$. 
\end{proof}

\subsection{A bijective case}

Let $r$ and $N$ be positive integers, $\Omega \subset \mathbb{R}^{r}$ be a $\mathscr{C}^{N}$-manifold and $\tau \colon \Omega \to \Omega$ be a $\mathscr{C}^{N}$-dif\-feo\-mor\-phism. Define $\tilde{T}_{\mathbf{0}}\colon \mathscr{C}^{N}(\Omega)\to \mathscr{C}(\Omega)$  through 
\[
 \tilde{T}_{\mathbf{0}}(f)(x)= f(\tau(x)) 
 \qquad 
 \left(x\in \Omega, f\in \mathscr{C}^{N}(\Omega)\right). 
\]
Then $\tilde{T}_{\mathbf{0}}$ is a multiplicative mapping. Thus it can be an appropriate candidate to generate a moment sequence on $\mathscr{C}^{N}(\Omega)$.

\begin{lem}
 Let $r$ and $N$ be positive integers, $\Omega \subset \mathbb{R}^{r}$ be a $\mathscr{C}^{N}$-manifold and $\tau \colon \Omega \to \Omega$ be a $\mathscr{C}^{N}$-diffeomorphism. 
 Further, let us consider the mappings $T_{\mathbf{0}},  \tilde{T}_{\mathbf{0}}\colon \mathscr{C}^{N}(\Omega)\to \mathscr{C}(\Omega)$ defined by 
 \[
T_{\mathbf{0}}(f)(x)= f(x) 
\qquad 
\text{and}
\qquad 
\tilde{T}_{\mathbf{0}} (f)(x)= f(\tau(x)) 
\qquad 
\left(x\in \Omega, f\in \mathscr{C}^{N}(\Omega)\right), 
 \]
 respectively. 
Then the following statements are equivalent: 
\begin{enumerate}[(i)]
 \item the sequence of mappings $T_{\alpha}\colon \mathscr{C}^{N}(\Omega)\to \mathscr{C}(\Omega)$, $\alpha\in \mathbb{N}^{r}, |\alpha|\leq N$ is a moment sequence generated by $T_{\mathbf{0}}$
 \item the sequence of mappings $\tilde{T}_{\alpha}\colon \mathscr{C}^{N}(\Omega)\to \mathscr{C}(\Omega)$, $\alpha\in \mathbb{N}^{r}, |\alpha|\leq N$ is a moment sequence generated by $\tilde{T}_{\mathbf{0}}$
\end{enumerate}
\end{lem}

\begin{proof}
 Let $r$ and $N$ be positive integers, $\Omega \subset \mathbb{R}^{r}$ be a $\mathscr{C}^{N}$-manifold and $\tau \colon \Omega \to \Omega$ be a $\mathscr{C}^{N}$-diffeomorphism. 
 Further, let is consider the mappings $T_{\mathbf{0}},  \tilde{T}_{\mathbf{0}}\colon \mathscr{C}^{N}(\Omega)\to \mathscr{C}(\Omega)$ defined by 
 \[
T_{\mathbf{0}}(f)(x)= f(x) 
\qquad 
\text{and}
\qquad 
\tilde{T}_{\mathbf{0}} (f)(x)= f(\tau(x)) 
\qquad 
\left(x\in \Omega, f\in \mathscr{C}^{N}(\Omega)\right), 
 \]
respectively. 

To prove the direction (i)$\Rightarrow$(ii), assume that the sequence of mappings $T_{\alpha}\colon \mathscr{C}^{N}(\Omega)\to \mathscr{C}(\Omega)$, $\alpha\in \mathbb{N}^{r}, |\alpha|\leq N$ is a moment sequence generated by $T_{\mathbf{0}}$. 
This means that for all $\alpha \in \mathbb{N}^{r}$, $|\alpha|\leq N$ we have 
\[
 T_{\alpha}(f \cdot g)(x)= \sum_{\beta \leq \alpha}\binom{\alpha}{\beta}T_{\beta}(f)(x)\cdot T_{\alpha-\beta}(g)(x)
\]
for all $f, g\in \mathscr{C}^{N}(\Omega)$ and $x\in \Omega$. Thus we also have 
\[
 T_{\alpha}(f \cdot g)(\tau(x))= \sum_{\beta \leq \alpha}\binom{\alpha}{\beta}T_{\beta}(f)(\tau(x))\cdot T_{\alpha-\beta}(g)(\tau(x))
 \qquad 
 \left(f, g\in \mathscr{C}^{N}(\Omega), x\in \Omega\right). 
\]
For all multi-indices $\alpha\in \mathbb{N}^{r}$, $|\alpha|\leq N$, define the mapping $\tilde{T}_{\alpha}\colon \mathscr{C}^{N}(\Omega)\to \mathscr{C}(\Omega)$ by 
\[
 \tilde{T}_{\alpha}(f)(x)= T_{\alpha}(f)(\tau(x))
 \qquad 
 \left(f\in \mathscr{C}^{N}(\Omega), x\in \Omega\right)
\]
to deduce that 
\[
 \tilde{T}_{\alpha}(f \cdot g)(x)= \sum_{\beta \leq \alpha}\binom{\alpha}{\beta}\tilde{T}_{\beta}(f)(x)\cdot \tilde{T}_{\alpha-\beta}(g)(x)
\]
for all $f, g\in \mathscr{C}^{N}(\Omega)$ and $x\in \Omega$. Thus the sequence of mappings $(\tilde{T}_{\alpha})_{|\alpha|\leq N}$ is a moment sequence of rank $r$ generated by $\tilde{T}_{0}$. 

The proof of the implication (ii)$\Rightarrow$(i) is analogous. It is enough to consider a point $x=\tau(y)$ with arbitrary $y \in \Omega$ and use the fact that $\tau$ is a diffeomorphism. 
\end{proof}

As we saw above, if $r$ and $N$ are positive integers, $\Omega \subset \mathbb{R}^{r}$ is a $\mathscr{C}^{N}$-manifold and $\tau \colon \Omega \to \Omega$ is a $\mathscr{C}^{N}$-dif\-feo\-mor\-phism, then the mapping $\tilde{T}_{\mathbf{0}}\colon \mathscr{C}^{N}(\Omega)\to \mathscr{C}(\Omega)$  defined by  
\[
 \tilde{T}_{\mathbf{0}}(f)(x)= f(\tau(x)) 
 \qquad 
 \left(x\in \Omega, f\in \mathscr{C}^{N}(\Omega)\right), 
\]
is a multiplicative mapping. Thus it can be an appropriate candidate to generate a moment sequence on $\mathscr{C}^{N}(\Omega)$. 
Nevertheless, the previous lemma says that instead of multiplicative mappings of this form, it suffices to consider the identity mapping.
Accordingly, below we will describe moment sequences generated by the identity mapping. Further, observe that while describing the solutions of equation \eqref{Eq_main}, not only the generator, i.e., the operator $T_{\mathbf{0}}$, but also the domain $\mathscr{C}^{N}(\Omega)$ can play a crucial role. In the second part of this section, we focus on the largest possible domain, that is, we will work on $\mathscr{C}(\Omega)$.

During the proof of Theorem \ref{thm2} we will use a corollary of \cite[Theorem 3.5]{KonMil18} and also \cite[Theorem 7.1]{KonMil18} which are the following statements. Before stating these results, however, we need two more notions from the theory of operator relations.

\begin{dfn}
 Let $k$ be a nonnegative integer, $r$ be a positive integer and $\Omega \subset \mathbb{R}^{r}$ be an open set. An operator $A\colon \mathscr{C}^{k}(\Omega)\to \mathscr{C}(\Omega)$ is \emph{non-degenerate} if for each nonvoid open subset $U\subset \Omega$ and all $x\in U$, there exist functions $g_{1}, g_{2}\in \mathscr{C}^{k}(\Omega)$ with supports in $U$ such that the vectors $(g_{i}(x), A g_{i}(x))\in \mathbb{R}^{2}$, $i=1, 2$ are linearly independent in $\mathbb{R}^{2}$. 
\end{dfn}

\begin{dfn}
 Let $k$ and $r$ be positive integers with $k\geq 2$ and $\Omega\subset \mathbb{R}^{r}$ be an open set. We say that the operator $A\colon \mathscr{C}^{k}(\Omega)\to \mathscr{C}(\Omega)$ \emph{depends non-trivially on the derivative} if there exists $x_0\in \Omega$ and there are functions $f_{1}, f_{2}\in \mathscr{C}^{k}(\Omega)$ such that 
 \[
  f_{1}(x_0)= f_{2}(x_0) \quad \text{and} \quad A f_{1}(x_0)\neq A f_{2}(x_0)
 \]
holds. 
\end{dfn}

\begin{prop}\label{prop_KonMil}
 Let $r$ be a positive integer and $\Omega\subset \mathbb{R}^{r}$ be a domain. Suppose that the operator $T\colon \mathscr{C}(\Omega)\to \mathscr{C}(\Omega)$ satisfies the Leibniz rule, i.e., 
 \[
  T(f\cdot g)= f\cdot T(g)+T(f)\cdot g 
  \qquad 
  \left(f, g\in \mathscr{C}(\Omega)\right). 
 \]
Then there exists a function $c\in \mathscr{C}(\Omega)$ such that for all $f\in \mathscr{C}(\Omega)$ and $x\in \Omega$ 
\[
 T(f)(x)= c(x)\cdot f(x) \cdot \ln \left(\left|f(x)\right|\right). 
\]
Conversely, any such map $T$ satisfies the Leibniz rule. 
\end{prop}

\begin{prop}\label{prop_KonMil2}
 Let $r$ be a positive integer, $k$ be a nonnegative integer and $\Omega \subset\mathbb{R}^{r}$ be a domain. Assume that $T, A\colon \mathscr{C}^{k}(\Omega)\to \mathscr{C}(\Omega)$ satisfy 
 \[
  T(f\cdot g)= T(f)\cdot g+f\cdot T(g)+2A(f)\cdot A(g) 
  \qquad 
  \left(f, g\in \mathscr{C}^{k}(\Omega)\right)
 \]
and that in case $k\geq 2$ the mapping $A$ is non-degenerate and depends non-trivially on the derivative. 
Then there are continuous functions $a\colon \Omega \to \mathbb{R}$ and $b, c\colon \Omega \to \mathbb{R}^{r}$ such that we have 
\[
 \begin{array}{rcl}
  T(f)(x)&=& \langle f''(x)c(x), c(x)\rangle +R(f)(x)\\[2mm]
  A(f)(x)&=& \langle f'(x), c(x)\rangle
 \end{array}
\qquad 
\left(f\in \mathscr{C}^{k}(\Omega), x\in \Omega\right), 
\]
where 
\[
 R(f)(x)= \langle f'(x), b(x)\rangle +a(x)f(x)\ln \left(|f(x)|\right) 
 \qquad 
 \left(f\in \mathscr{C}^{k}(\Omega)\right). 
\]
If $k=1$, then necessarily $c\equiv 0$. Further, if $k=0$, then necessarily $b\equiv 0$ and $c\equiv 0$.

Conversely, these operators satisfy the above second-order Leibniz rule. 
\end{prop}

Our main result for operators defined on $\mathscr{C}(\Omega)$ is the following theorem.

\begin{thm}\label{thm2}
 Let $r$ and $N$ be positive integers, $\Omega \subset \mathbb{R}^{r}$ be a domain and assume that for all multi-indices $\alpha\in \mathbb{N}^{r}$, with $|\alpha|\leq N$ we are given a mapping $T_{\alpha}\colon \mathscr{C}(\Omega)\to \mathscr{C}(\Omega)$ such that 
 $T_{\mathbf{0}}$ is the identity mapping and for all multi-indices $\alpha \in \mathbb{N}^{r}$ with $0\neq |\alpha|\leq N$ we have 
 \begin{equation}\label{moment_C}
  T_{\alpha}(f\cdot g)= \sum_{\beta \leq \alpha} \binom{\alpha}{\beta}T_{\beta}(f)\cdot T_{\alpha-\beta}(g)
 \end{equation}
 for all $f, g\in \mathscr{C}(\Omega)$. Then  there exist a family  of functions $\{ c_{\alpha}\in \mathscr{C}(\Omega) :  0\neq |\alpha | \leq N \}$ such that
\begin{equation}\label{c}
 \left[\sum_{\mathbf{0} \lneq \beta \lneq \alpha}\binom{\alpha}{\beta} c_{\beta}(x)\cdot c_{\alpha-\beta }(x)\right] =0
\end{equation}
 and
\begin{equation}\label{f}
  T_{\alpha}(f)(x)= c_{\alpha}(x)f(x)\ln(|f(x)|) 
  \qquad 
  \left(x\in \Omega, f\in \mathscr{C}(\Omega),  0\neq |\alpha | \leq N \right). 
\end{equation}
And also conversely, if $T_{\mathbf{0}}$ is the identity mapping on $\mathscr{C}(\Omega)$, we are given a family of functions that satisfies \eqref{c} and we define the mappings $T_{\alpha}$ on $\mathscr{C}(\Omega)$ by the formula \eqref{f}, then they satisfy equation \eqref{moment_C} for all multi-indices $\alpha$ such that  $0\neq |\alpha | \leq N$. 
\end{thm}

\begin{proof}
  Let $r$ and $N$ be positive integers, $\Omega \subset \mathbb{R}^{r}$ be a domain and assume that for all multi-indices $\alpha\in \mathbb{N}^{r}$, with $|\alpha|\leq N$ we are given a mapping $T_{\alpha}\colon \mathscr{C}(\Omega)\to \mathscr{C}(\Omega)$ such that 
 $T_{\mathbf{0}}$ is the identity mapping and for all multi-indices $\alpha \in \mathbb{N}^{r}$ with $0\neq |\alpha|\leq N$ we have 
 \[
  T_{\alpha}(f\cdot g)= \sum_{\beta \leq \alpha} \binom{\alpha}{\beta}T_{\beta}(f)\cdot T_{\alpha-\beta}(g)
 \]
 for all $f, g\in \mathscr{C}(\Omega)$. 
 
 We prove the statement by induction on the multi-index $\alpha$. 
 
 Let $\alpha \in \mathbb{N}^{r}$ be an arbitrary multi-index for which $|\alpha|=1$ holds. Then 
 \[
  T_{\alpha}(f\cdot g)=T_{\mathbf{0}}(f)T_{\alpha}(g)+T_{\alpha}(f)T_{\mathbf{0}}(g)=f \cdot T_{\alpha}(g)+T_{\alpha}(f) \cdot g
  \qquad 
  \left(f, g\in \mathscr{C}(\Omega)\right), 
 \]
since $T_{\mathbf{0}}= \mathrm{id}$ was assumed. Using Proposition \ref{prop_KonMil}, we obtain that there exists a continuous function $c_{\alpha}\in \mathscr{C}(\Omega)$ such that 
\[
  T_{\alpha}(f)(x)= c_{\alpha}(x)f(x)\ln(|f(x)|) 
  \qquad 
  \left(x\in \Omega, f\in \mathscr{C}(\Omega)\right). 
 \]
 
 Let now $k\in \left\{ 1, \ldots, N-1\right\}$ be arbitrary and assume that the statement of the theorem holds for all multi-indices $\beta\in \mathbb{N}^{r}$ for which we have $|\beta|\leq k$. Let further $\alpha\in \mathbb{N}^{r}$ be an arbitrary multi-index for which $|\alpha|=k+1$. 
 Then 
 \begin{align*}
  T_{\alpha}(f\cdot g) &= \sum_{\beta \leq \alpha}\binom{\alpha}{\beta}T_{\beta}(f)\cdot T_{\alpha-\beta}(g) \\
  &= T_{\mathbf{0}}(f)\cdot T_{\alpha}(g)+T_{\alpha}(f)\cdot T_{\mathbf{0}}(g)+ \sum_{\mathbf{0} \lneq \beta \lneq \alpha}\binom{\alpha}{\beta}T_{\beta}(f)\cdot T_{\alpha-\beta}(g) \\
  &= f\cdot T_{\alpha}(g)+T_{\alpha}(f)\cdot g+ \sum_{\mathbf{0} \lneq \beta \lneq \alpha}\binom{\alpha}{\beta} c_{\beta}f\ln(|f|)  \cdot c_{\alpha-\beta }g\ln(|g|)  \\
  &= f\cdot T_{\alpha}(g)+T_{\alpha}(f)\cdot g+ \left[\sum_{\mathbf{0} \lneq \beta \lneq \alpha}\binom{\alpha}{\beta} c_{\beta}(x)\cdot c_{\alpha-\beta }\right] \cdot f \ln(|f|) \cdot  g(x)\ln(|g|)  \\
 \end{align*}
holds for all $f, g\in \mathscr{C}(\Omega)$. Using Proposition \ref{prop_KonMil2}, taking into account that $k=0$, we obtain that there exists a continuous function $c_{\alpha}$ such that 
\[
 T_{\alpha}(f)(x)= c_{\alpha}(x)\cdot f(x) \cdot \ln(|f(x)|)
\]
is fulfilled for all $f\in \mathscr{C}(\Omega)$ and $x\in \Omega$. 
Further, the family of functions  $\{ c_{\alpha}\in \mathscr{C}(\Omega) :  0\neq |\alpha | \leq N \}$ necessarily satisfies \eqref{c}.

The converse implication is an easy computation.
\end{proof}

As we saw in the previous theorem, the moment sequences are quite poor on the $\mathscr{C}(\Omega)$ space. We note that if $N\geq 1$, then there are substantially more diverse moment sequences in the space $\mathscr{C}^{N}(\Omega)$, see Remark \ref{rem3}. However, this will be dealt with in one of our future work.

\begin{ackn}
 The research of Eszter Gselmann has been supported by project no.~K134191 that has been implemented with the support provided by the National Research, Development and Innovation Fund of Hungary,
financed under the K\_20 funding scheme. 

The work of Aleksandra \'Swi\k{a}tczak is implemented under the project "Curriculum for advanced doctoral education \&
training - CADET Academy of TUL" co-financed by the STER Programme - Internationalization of
doctoral schools.

This article has been completed while one of the authors (Aleksandra \'Swi\k{a}tczak), was the
Doctoral Candidate in the Interdisciplinary Doctoral School at the Lodz University of Technology, Poland.
\end{ackn}



\end{document}